\documentclass[10pt,a4paper,oneside,final,reqno]{amsart}

\usepackage[backend=bibtex, style=numeric, maxnames=50]{biblatex}
\renewbibmacro{in:}{}

\DeclareFontFamily{OT1}{pzc}{}
\DeclareFontShape{OT1}{pzc}{m}{it}{<-> s * [1.10] pzcmi7t}{}
\DeclareMathAlphabet{\mathpzc}{OT1}{pzc}{m}{it}

\allowdisplaybreaks

\usepackage{comment}
\usepackage[T1]{fontenc}
\usepackage[utf8]{inputenc}
\usepackage[dvipsnames]{xcolor}
\usepackage[english]{babel}
\usepackage{mathtools, amsthm, amssymb, amscd, amsmath} 
\usepackage{newtxtext, newtxmath} 
\usepackage[DIV=11]{typearea} 
\usepackage{hyperref} 
\usepackage[notref,notcite]{showkeys} 
\usepackage{todonotes}
\usepackage{enumitem} 
\usepackage{xypic} 
\usepackage{tikz}
\usetikzlibrary{cd, decorations.markings, arrows} 

\usepackage[capitalise]{cleveref}

\DeclareMathAlphabet{\mathcal}{OMS}{cmsy}{m}{n} 

\definecolor{DarkPurple}{rgb}{0.40,0.0,0.20}
\hypersetup{
	colorlinks =true,
	linkcolor= DarkPurple, 
	citecolor = NavyBlue 
} 

\usetikzlibrary{cd, decorations.markings, arrows, matrix, calc}

\newcommand{\G}{\mathcal{G}}

\newcommand{\A}{\mathcal{A}}

\newcommand{\Cfull}{C^*}

\newcommand{\C}{\mathbb{C}}
\newcommand{\R}{\mathbb{R}}

\newcommand{\Z}{\mathbb{Z}}
\newcommand{\N}{\mathbb{N}}
\newcommand{\K}{\mathbb{K}}
\newcommand{\HS}{\mathscr{H}}

\newcommand{\lt}{\triangleleft}
\newcommand{\rt}{\triangleright}

\DeclareMathOperator{\supp}{supp}

\newtheorem{lemma}{Lemma}[section]
\newtheorem{corollary}[lemma]{Corollary}
\newtheorem{theorem}[lemma]{Theorem}
\newtheorem*{theorem*}{Theorem}
\newtheorem{proposition}[lemma]{Proposition}

\theoremstyle{definition}

\newtheorem{example}[lemma]{Example}

\title[]{Exotic groupoid C*-algebras associated to double groupoids}

\author[1]{Mathias Palmstrøm}

\address{Department of Mathematical Sciences, Faculty of Information Technology and Electrical Engineering, NTNU -- Norwegian University of Science and Technology, Trondheim, Norway}
\email{mathias.palmstrom@ntnu.no}

\numberwithin{equation}{section} 

\begin{document}
	\renewcommand{\thefootnote}{}
	
\begin{abstract}
	We consider a class of partial action groupoids called double groupoids which are constructed from pairs of subgroups satisfying similar conditions to those of a matched pair of groups. If the double groupoid is étale, then we show that whenever the partially acting group admit exotic ideal completions in the sense of Brown and Guentner, the corresponding double groupoid also admit exotic $C^*$-completions.
\end{abstract}
\maketitle

\section{Introduction}
An exotic groupoid $C^*$-algebra associated to a locally compact Hausdorff groupoid $\G$ admitting a Haar system is a $C^*$-completion $C_{e}^{\ast}(\G) = \overline{C_c (\G)}^{\Vert \cdot \Vert_e}$, where $\Vert \cdot \Vert_e$ is a $C^*$-norm on $C_c (\G)$ which dominates the reduced norm and differs from both the full and reduced norms. Such $C^*$-norms need not exist, for example if the groupoid has the weak containment property. It is a natural and interesting problem of finding examples of such $C^*$-algebras, one which has received much attention in the last decade or so, particularly in the context of groups \cite{BekkaEtAl:OnC*AlgebrasAssociatedWithLocallyCompactGroups, BrownandGuentner:NewC*-completions, deLaatEtAl:GroupC*AlgebrasOfLocallyCompactGroupsActingOnTrees, LandstadEtAl:ExoticGroupC*AlgebrasInNoncommutativeDuality, deLaatAndSiebenand:ExoticGroupC*AlgebrasOfSimpleLieGroupsWithRealRankOne, Okayasu:FreeGroupC*-algebrasAssociatedWithLP, RuanAndWiersma:OnExoticGroupC*Algebras, SameiAndWiersma:ExoticC*AlgebrasOfGeometricGroups, Wiersma:ConstructionsOfExoticGroupC*Algebras} and crossed products \cite{BussetAL:ExoticCrossedProducts, BussetAL:ExoticCrossedProductsAndTheBaumConnesConjecture, ExelPittsZarikian:ExoticIdealsInFreeTransformationGroupC*Algebras}, but also also for quantum groups \cite{BrannanAndRuan:L_PRepresentationsOfDiscreteQuantumGroups, KyedAndSoltan:PropertyTAndExoticQuantumGroupNorms}. For groupoids which are not necessarily groups or crossed products, far less examples are known. At least exotic groupoid $C^*$-algebras associated to principal twisted étale groupoids were characterized by Exel in \cite{Exel:OnKumjiansC*DiagonalsAndTheOpaqueIdeal} as pairs of inclusions $(A,B)$ where $B$ is a $C^*$-algebra and $A$ is a closed commutative $^*$-subalgebra of $B$ which is regular and satisfies the extension property. See also \cite[Theorem 3.11.6]{ExelAndPitts:CharacterizingGroupoidC*AlgebrasOfNonHausdorffEtaleGroupoids} for a similar result with topologically principal instead of principal. Note that there the full and reduced $C^*$-algebras are also considered exotic. 
In \cite{BruceAndLi:AlgebraicActionsI.C*-algebrasAndGroupoids}, Bruce and Li construct an inverse semigroup from an algebraic action $\sigma \colon S \curvearrowright G$ of a semigroup $S$ on a group $G$. Letting $\mathfrak{A}_\sigma$ denote the concrete $C^*$-algebra generated by the Koopman representation for the action together with the left regular representation, they found sufficient conditions for when $\mathfrak{A}_\sigma$ is an exotic groupoid $C^*$-algebra for the tight groupoid $\G_\sigma$ built from the inverse semigroup. 
In \cite{Palmstrøm:ExoticCstarCompletionsOfEtaleGroupoids}, the author generalized the construction of Brown and Guentner in \cite{BrownandGuentner:NewC*-completions} to second countable étale groupoids, and used this construction to show that certain hyperbolic groupoids (see \cite[Definition 5.2]{Palmstrøm:ExoticCstarCompletionsOfEtaleGroupoids}) admits exotic $C^*$-completions. 

In this paper, we study the problem of finding exotic $C^*$-completions of a class of groupoids called double groupoids. The name stems from the fact that the underlying locally compact Hausdorff space of such a groupoid admits two different groupoid structures. Double groupoids are constructed from a pair of subgroups satisfying conditions similar to those of a matched pair of groups (see for example \cite{BaajSkandalisVaes:NonSemiRegularQuantumGroupsComingFromNumberTheory, DesmedQuaegebeurVaes:AmenabilityAndTheBicrossedProductConstruction, VaesVainerman:ExtensionsOfLocallyCompactGroupsAndTheBicrossedProductConstruction}). They have been considered in \cite{LandstadDaele:FiniteQuantumHypergroups} where the pair of groups from which the groupoid is constructed are finite and discrete, so that the same is the case for the double groupoids. Therein, the two double groupoids sharing the same underlying set gives rise to a dual pair of quantum hypergroups (see \cite{DelvauxDaele:AlgebraicQuantumHypergroups, DelvauxDaele:AlgebraicQuantumHypergroupsIIConstructiosAndExamples}). The non-finite case of \cite{LandstadDaele:FiniteQuantumHypergroups} is to be covered in \cite{LandstadDaele:TopQuantumHypergroupsArisingFromTwoClosedSubgroups}.
In general, double groupoids are a particular instance of local action groupoids as described in \cite[Definition 1.1 and Proposition 1.2]{LandstadVanDaele:PolynomialFunctionsForLocallyCompactGroupActions}, or equivalently partial transformation groupoids as described in \cite{Abadie:OnPartialActionsAndGroupoids}. Our main result, \cref{thm: exotic group C*-algebras induce exotic groupoid C*-algebras}, is in the context of étale double groupoids, and states that if the partially acting group admit exotic ideal completions as described in \cite{BrownandGuentner:NewC*-completions}, then the double groupoid admit exotic $C^*$-completions as described in \cite{Palmstrøm:ExoticCstarCompletionsOfEtaleGroupoids}. Along the way, we make some observations regarding some of the basic properties of double groupoids. Perhaps the most interesting of these is that an étale second countable double groupoid is amenable if and only if it has the weak containment property (\cref{prop: amenability equivalent to WCP for this groupoid}).

\ 
\newline

The paper is organized as follows: \cref{sec: prelims} introduces étale groupoids and their unitary representations, and recalls some notions and results which will be used throughout. In \cref{sec: Defs and elementary props}, we give the definition of a double groupoid and outline some of its basic properties. Some concrete examples are given in \cref{sec: concrete examples}. Finally, in \cref{sec: Existence of exotic groupoid C*-algebras}, we find sufficient conditions for the existence of exotic groupoid $C^*$-algebras associated to étale double groupoids.

\section*{Acknowledgements}
I would like to thank Magnus B. Landstad for introducing me to double groupoids and showing me some of the concrete examples that appear in this paper, and I also thank both him and Alfons Van Daele for sharing with me their unpublished paper \cite{LandstadDaele:TopQuantumHypergroupsArisingFromTwoClosedSubgroups}. I am also grateful to the anonymous referee for pointing out that \cref{ex: transformation groupoids from Zappa-Szep products} could be extended to its present more general form.
Finally, I would like to thank Eduard Ortega for some valuable discussions.

\section{Preliminaries on Groupoids and their Unitary Representations} \label{sec: prelims}
The purpose of this section is to recall the main definitions and results regarding groupoids and their unitary representations, as well as to establish some notation (which will be the same as in \cite{Palmstrøm:ExoticCstarCompletionsOfEtaleGroupoids}). The basic references here are \cite{Muhly:CoordinatesInOperatorAlgebras, Paterson:GroupoidsAndTheirOperatorAlgebras, Renault:AGroupoidApproachToC*Algebras, Sims:EtaleGroupoids, Williams:AToolKitForGroupoidCstarAlgebras, Williams:CrossedProductsOfC*Algebras}, and the reader is referred to these for more details.

\ 

Let $\G$ be a groupoid with unit space $\G^{(0)}$ and range and source maps $r \colon \G \to \G \, , \, x \mapsto xx^{-1}$ and $s \colon \G \to \G \, , \, x \mapsto x^{-1} x$. We shall always at the very least assume that $\G$ is endowed with a locally compact Hausdorff topology for which the inversion and multiplication are continuous. Usually we will in addition require the topology to be second countable, and such that the range map is a local homeomorphism, in which case the groupoid is called \emph{étale}. These assumptions will always be stated explicitly whenever they are needed. For any $X \subset \G^{(0)} ,$ we denote by $\G_X = \{x \in \G \colon s(x) \in X\} $ and $\G^{X} = \{x \in \G \colon r(x) \in X\}$. We shall write $\G_u$ and $\G^{u}$ instead of $\G_{\{u\}}$ and $\G^{\{u\}}$, whenever $u \in \G^{(0)}$ is a unit. The \emph{isotropy group} at a unit $u \in \G^{(0)}$ is the group $\G_{u}^{u} := \G(u) := \G_u \cap \G^{u}$. If $\G(u) = \{u\}$, for all $u \in \G^{(0)}$, then $\G$ is called \emph{principal}. $\G$ is called \emph{topologically principal} when the set $\{u \in \G^{(0)} \colon \G(u) = \{u\} \}$ is dense in $\G^{(0)}$. A subset $X \subset \G^{(0)}$ is \emph{invariant} if for all $x \in \G$, $s(x) \in X$ if and only if $r(x) \in X$. If for every $u \in \G^{(0)}$ the set $\{r(x) \colon s(x) = u \}$ is dense in $\G^{(0)}$, then $\G$ is said to be \emph{minimal}. 
A \emph{homomorphism} between two étale groupoids $\G$ and $\mathcal{H}$ is a continuous map $\phi \colon \G \to \mathcal{H}$ such that if $(x,y) \in \G^{(2)}$, then $(\phi(x) , \phi(y)) \in \mathcal{H}^{(2)}$, and in this case $\phi(xy) = \phi(x) \phi(y)$; if $\phi$ here still preserves the algebraic structure but is only a Borel map, then we call it a \emph{Borel homomorphism}.

The class of groupoids that we will study in this note is a particular instance of a \emph{local action groupoid} as described in \cite[Definition 1.1 and Proposition 1.2]{LandstadVanDaele:PolynomialFunctionsForLocallyCompactGroupActions}. Let us recall these: Assume that $G$ is a locally compact group and that $X$ is a locally compact Hausdorff space. A \emph{local left action} of $G$ on $X$ is a continuous map $$ \Omega \to X \, , \, (g,x) \mapsto g \cdot x ,$$ defined on an open set $\Omega \subset G \times X$ such that $(e,x) \in \Omega $ and $e \cdot x = x$ for all $x \in X$, and if $(h,x) \in \Omega$, then $(g,h \cdot x) \in \Omega$ if and only if $(gh,x) \in \Omega$, in which case $g \cdot (h \cdot x) = (gh) \cdot x$. A local right action is defined similarly. From such an action one can construct a locally compact Hausdorff groupoid $\G(\Omega)$ which as a topological space is just $\Omega$, and with operations given as follows: $(g,x)^{-1} = (g^{-1} , g \cdot x)$ and $(g,x) \cdot (h,y)$ is defined whenever $x = h \cdot y$, in which case $(g,h \cdot y) \cdot (h,y) = (gh,y)$. The unit space can be identified with $X$ via the identification $(e,x) \mapsto x$, and under this identification, the range and source maps become $r(g,x) = g \cdot x$ and $s(g,x) = x$. The groupoid associated to a local right action is defined analogously. If $\G(\Omega)$ is the groupoid constructed from a local left action of $G$ on $X$, then it is easy to see that $\G(\Omega)$ is étale if and only if $G$ is discrete. If we define $X_g := \{x \in X \colon (g,x) \in \G(\Omega) \}$, for each $g \in G$, then it is not hard to see that each set $X_g$ is open, and that the maps $\theta_g (x) = g \cdot x$ are homeomorphisms from $X_g $ to $X_{g^{-1}}$ satisfying $\theta_g \circ \theta_h \subset \theta_{gh}$; so  $\G (\Omega)$ is nothing but a \emph{partial transformation groupoid} as described in \cite{Abadie:OnPartialActionsAndGroupoids} formed from the \emph{partial dynamical system} $\theta = (\{X_g\}_{g \in G}, \{\theta_g\}_{g \in G})$. Conversely, given a locally compact group $G$, a locally compact Hausdorff space $X$ and a partial dynamical system $\theta = (\{X_g\}_{g \in G}, \{\theta_g\}_{g \in G})$ as in \cite[Definition 1.1]{Abadie:OnPartialActionsAndGroupoids}, the map $(g,x) \mapsto \theta_g (x)$ defined on the open subset $\{(g,x) \in G \times X \colon x \in X_{g} \} \subset G \times X$, defines a local action of $G$ on $X$, and the associated local action groupoid is precisely the partial transformation groupoid formed from the partial dynamical system.

\

Let us assume for the remainder of the section that $\G$ is an étale groupoid. The space of continuous compactly supported functions $C_c (\G)$ becomes a normed $^*$-algebra in the following manner: first of all, since $\G$ is étale, the fibers $\G_u$ and $\G^u$, for $u \in \G^{(0)}$, are discrete. We endow $C_c (\G)$ with the convolution product, which for $f,g \in C_c (\G)$ is given by $$ f \ast g (x) = \sum_{y \in \G_{s(x)}} f(x y^{-1}) g(y) = \sum_{y \in \G^{r(x)}} f(y) g(y^{-1}x) ,$$ for $x \in \G$. The involution is defined by $ f^{\ast}(x) = \overline{f(x^{-1})} $, for $f \in C_c (\G)$ and $x \in \G$, and the \emph{$I$-norm} on $C_c (\G)$ is given by $$ \| f \|_I = \max \left\lbrace \sup_{u \in \G^{(0)}} \sum_{x \in \G_u} |f(x)| \, , \, \sup_{u \in \G^{(0)}} \sum_{x \in \G^{u}} |f(x)| \right\rbrace .$$ With the above norm and algebraic operations, $(C_c (\G), \ast, ^{\ast}, \| \cdot \|_{I})$ becomes an involutive normed $^*$-algebra.

\

Let $\{\HS(u)\}_{u \in \G^{(0)}}$ be a collection of Hilbert spaces indexed by $\G^{(0)}$. The set $\G^{(0)} \ast \HS := \{ (u,\xi) \colon \xi \in \HS (u) \}$ is said to be a \emph{Borel Hilbert bundle} if it has a standard Borel structure such that 

\begin{itemize}
	\item[(i)] $E \subset \G^{(0)}$ is Borel if and only if $p^{-1}(E)$ is Borel in $\G^{(0)} \ast \HS$, where $p \colon \G^{(0)} \ast \HS \to \G^{(0)}$ is the projection $p(u, \xi) = u$, for $(u,\xi) \in \G^{(0)} \ast \HS$;
	\item[(ii)] there is a sequence of sections $\{f_n \}_n$, called a \emph{fundamental sequence}, such that
	\begin{itemize}
		\item[(a)] for each $n$, the map $\hat{f}_n \colon \G^{(0)} \ast \HS \to \C$, given by $\hat{f}_n (u, \xi) = \langle f_n (u) , \xi \rangle_{\HS(u)}$, is Borel;
		\item[(b)] for each $n, m$, the map $u \mapsto \langle f_n (u) , f_m (u) \rangle_{\HS(u)}$, is Borel;
		\item[(c)] the sequence of functions $\{\hat{f}_n\}_n$ together with $p$, separate points of $\G^{(0)} \ast \HS$.
	\end{itemize}
\end{itemize} 

Given a Borel Hilbert bundle $\G^{(0)} \ast \HS$ we can associate a standard Borel groupoid called the \emph{isomorphism groupoid} as follows: As a set, the groupoid is $$\mathrm{Iso}(\G^{(0)} \ast \HS) = \left\lbrace (u,T,v) \colon u,v \in \G^{(0)} \text{ and } T \in \mathcal{U}(\HS(v) , \HS(u)) \right\rbrace ,$$ where $\mathcal{U}(\HS(v) , \HS(u))$ denotes the collection of Hilbert space isomorphisms between $\HS(v)$ and $\HS(v)$. The operations are given by $(u,T,v)^{-1} = (v,T^{-1},u)$ and $(u,T,v) (v,S,w) = (u,TS,w)$. If $\{f_n\}_{n}$ is a fundamental sequence for $\G^{(0)} \ast \HS$, then when endowed with the weakest Borel structure for which the maps $(u,T,v) \mapsto \langle T f_n (v) , f_m (u) \rangle_{\HS(u)}$, for $m,n$, are all Borel, $\mathrm{Iso}(\G^{(0)} \ast \HS)$ becomes a standard Borel space, and the groupoid operations are Borel.

A \emph{unitary representation} of an étale groupoid $\G$ is a Borel homomorphism $\pi \colon \G \to \text{Iso}(\G^{(0)} \ast \HS_\pi)$ which preserves the unit space, in the sense that for all $x \in \G$, $\pi(x) = (r(x) , \hat{\pi}(x) , s(x))$, where $\hat{\pi}(x) \in \mathcal{U}(\HS_\pi (s(x)) , \HS_\pi (r(x)))$. We usually identify $\pi$ with $\hat{\pi}$.

An important example is the \emph{left regular representation} of $\G$, which we denote by $\lambda$. Here $\HS_\lambda (u) = \ell^{2}(\G^u)$, for each $u \in \G^{(0)}$, and $\lambda \colon \G \to \text{Iso}(\G^{(0)} \ast \HS_\lambda)$ is given by $$\lambda(x) \colon \ell^{2}(\G^{s(x)}) \to \ell^2 (\G^{r(x)}) \, , \, \lambda(x) \xi (y) = \xi (x^{-1} y).$$ A fundamental sequence for the Hilbert Bundle $\G^{(0)} \ast \HS_\lambda$ can for example be any countable sup-norm dense sequence $\{f_n\}_n$ in $C_c (\G)$, and identifying any such $f_n$ with the section $u \mapsto {f_n}_{|_{G^{u}}} \in \ell^{2}(\G^{u})$. 

\ 

Let $\mu$ be a Radon measure on $\G^{(0)}$. The positive linear functional $$ f \mapsto \int_{\G^{(0)}} \sum_{x \in \G^{u}} f(x) \, d \mu (u) $$ corresponds uniquely to a Radon measure $\nu$ on $\G$ such that $$ \int_{\G^{(0)}} \sum_{x \in \G^{u}} f(x) \, d \mu (u) = \nu(f) ,$$ for all $f \in C_c (\G)$. If $\nu^{-1}$ denotes the push-forward measure under inversion, then we have that $$ \int_{\G^{(0)}} \sum_{x \in \G_{u}} f(x) \, d \mu (u) = \nu^{-1}(f) ,$$ for all $f \in C_c (\G)$. The Radon measure $\mu$ is said to be \emph{quasi-invariant} if $\nu$ and $\nu^{-1}$ are mutually absolutely continuous. 
In this case, the Radon-Nikodym theorem gives a Borel homomorphism $\Delta = \frac{d \nu}{d \nu^{-1}}$ from $\G$ to the multiplicative group of positive real numbers, which we call the \emph{modular function} associated with $\mu$ (see \cite{Hahn:HaarMeasureForMeasuredGroupoids,Ramsay:TopologiesOnMeasuredGroupoids}).

Given a Radon measure $\mu$ on $\G^{(0)}$, and a Borel Hilbert bundle over $\G^{(0)}$, the collection of Borel sections $f \colon \G^{(0)} \to \G^{(0)} \ast \HS$ such that the map $u \mapsto \Vert f(u) \Vert_{\HS(u)}^{2}$ is $\mu$-integrable, is naturally a pre-inner product space under the pre-inner product $$ \langle f,g \rangle := \int_{\G^{(0)}} \langle f(u) , g(u) \rangle_{\HS(u)} \, d \mu(u) .$$ Upon separating and completing, we obtain the Hilbert space known as the \emph{direct integral}, often denoted by $\int_{\G^{(0)}}^{\oplus} \HS(u) \, d\mu(u)$. If $\mu$ is a quasi-invariant measure on $\G^{(0)}$ and $\pi$ is a unitary representation of $\G$, with associated Borel Hilbert Bundle $\G^{(0)} \ast \HS_\pi$, then $\pi$ integrates to an $I$-norm bounded representation of $C_c (\G)$ on $\int_{G^{(0)}}^{\oplus} \HS_\pi (x) \, d \mu(x)$, denoted $\pi_\mu$, such that $$ \langle \pi_\mu (f) \xi , \eta \rangle = \int_{G} f(x) \langle \pi(x) \xi(s(x)) , \eta(r(x)) \rangle_{\HS_\pi(r(x))} \Delta^{-1/2} (x) \, d \nu (x) ,$$ for $\xi , \eta \in \int_{\G^{(0)}}^{\oplus} \HS_\pi(x) \, d \mu(x)$. The representation $\pi_\mu$ of $C_c (\G)$ is called the \emph{integrated form} of $\pi$ with respect to the quasi-invariant measure $\mu$. We shall sometimes denote by $\HS_{\pi,\mu}$ the direct integral $\int_{\G^{(0)}}^{\oplus} \HS_\pi(u) \, d \mu(u)$. Conversely, every representation of $C_c (\G)$ on a separable Hilbert space is unitarily equivalent to $\pi_\mu$ for some unitary representation $\pi$ on $\G$ and quasi-invariant measure $\mu$ (see \cite[Theorem 1.21]{Renault:AGroupoidApproachToC*Algebras}).

\

Recall that the \emph{full $C^*$-algebra} associated to the groupoid $\G$ is $\Cfull(\G) := \overline{C_c (\G)}^{\Vert \cdot \Vert_{max}} ,$ where for any $f \in C_c (\G)$, $ \Vert f \Vert_{max} := \sup_{\mu , \pi} \Vert \pi_\mu (f) \Vert $, where the supremum is taken over all unitary representations $\pi$ of the groupoid and all quasi-invariant measures $\mu$ on $\G^{(0)}$. The \emph{reduced $C^*$-algebra} associated to $\G$ is $C_{r}^{\ast}(\G) := \overline{C_c (\G)}^{\Vert \cdot \Vert_{r}},$ where for any $f \in C_c (\G)$, $\Vert f \Vert_{r} := \sup_{\mu} \Vert \lambda_\mu (f) \Vert$ where the supremum is taken over all quasi-invariant measures $\mu$ on $\G^{(0)}$. 

Let $B(\G)$ denote the commutative algebra of bounded Borel functions on $\G$. Suppose that $D \trianglelefteq B(\G)$ is an algebraic ideal. A unitary representation $\pi$ of $\G$ is a \emph{$D$-representation} if there exists a fundamental sequence $\{f_n\}_n$ of the Borel Hilbert bundle $\G^{(0)} \ast \HS_\pi$, such that for all $n, m \in \N$, the Borel function $$ x \mapsto \langle \pi(x) f_n(s(x)), f_m (r(x)) \rangle_{\HS_\pi (r(x))} ,$$ is an element of $D$. 
Given an ideal $D \trianglelefteq B(\G)$ and a family of quasi-invariant measures $\mathcal{M}$, one can associate a groupoid $C^*$-algebra as follows: Define a $C^*$-seminorm on $C_c (\G)$ by $$\| f \|_{D, \mathcal{M}} := \sup\left\lbrace \| \pi_\mu (f) \| \colon \, \pi \text{ is a } D\text{-representation and } \mu \in \mathcal{M} \right\rbrace.$$ Putting $\mathcal{N}_{D , \mathcal{M}} := \left\lbrace f \in C_c (\G) \colon \| f \|_{D , \mathcal{M}} = 0 \right\rbrace$, we define $$C_{D , \mathcal{M}}^{\ast} (\G) := \overline{C_c (\G)/\mathcal{N}_{D , \mathcal{M}}}^{\| \cdot \|_{D , \mathcal{M}}} .$$
This above construction naturally encapsulates the full and reduced $C^*$-algebras in that for a second countable Hausdorff étale groupoid $\G$, one has that $C^{\ast} (\G) = C_{B(\G)}^{\ast}(\G)$ and $C_{r}^{\ast} (\G) = C_{B_c (\G)}^{\ast} (\G)$, where $B_c (\G)$ denotes the ideal of bounded compactly supported Borel functions. In fact, $C_{r}^{\ast} (\G) = C_{B_c (\G) , \mu}^{\ast} (\G)$ when $\mu$ is quasi-invariant with full support (see \cite[Proposition 3.6, Proposition 3.16]{Palmstrøm:ExoticCstarCompletionsOfEtaleGroupoids}).

\section{Double Groupoids} \label{sec: Defs and elementary props}
We proceed to construct the class of groupoids of interest in this paper. Let $G$ be a locally compact Hausdorff group, and let $H,K$ be two closed subgroups such that $H \cap K = \{e\}$, where $e$ is the identity element in $G$. Assume moreover that $H K \subset G$ is open and that the map $$ H \times K \to H K \, , \, (h,k) \mapsto hk ,$$ is a homeomorphism. We shall call such a pair of closed subgroups $(H,K)$ satisfying the above an \emph{admissible pair}.
Notice that whenever $hk \in K H$, then there exists $h^\prime \in H$ and $k^\prime \in K$ such that $hk = k^\prime h^\prime$. If $\bar{k} \in K$ and $\bar{h} \in H$ were also such that $hk = k^\prime h^\prime = \bar{k} \bar{h}$, then $\bar{k}^{-1} k^\prime = \bar{h}^{-1} h^\prime \in H \cap K =\{e\}$, whence $k^\prime = \bar{k}$ and $h^\prime = \bar{h}$. So these $h^\prime$ and $k^\prime$ are unique. We put $h \triangleleft k := h^\prime$ and $h \triangleright k := k^\prime$, so that $hk = (h \triangleright k) (h \triangleleft k)$. It is straightforward to verify the following identities:
If $h_2 k \in KH$ then $h_1 (h_2 \rt k) \in KH$ if and only if $(h_1 h_2) k \in KH$, in which case
\begin{equation}\label{eq: extension of product in terms of H 1}
	h_1 \triangleright (h_2 \triangleright k) = (h_1 h_2) \triangleright k,
\end{equation}
and 
\begin{equation}\label{eq: extension of product in terms of H 2}
	(h_1 h_2) \triangleleft k = (h_1 \triangleleft (h_2 \triangleright k)) (h_2 \triangleleft k).
\end{equation}
Similarly, if $hk_1 \in KH$ then $(h \lt k_1)k_2 \in KH$ if and only if $h (k_1 k_2) \in KH$, in which case
\begin{equation}\label{eq: extension of product in terms of K 1}
	(h \lt k_1) \triangleleft k_2 = h \triangleleft (k_1 k_2),
\end{equation}
and
\begin{equation}\label{eq: extension of product in terms of K 2}
	h \rt (k_1 k_2) = (h \triangleright k_1) ((h \triangleleft k_1) \triangleright k_2).
\end{equation}
Moreover, 
\begin{equation} \label{eq: identity element is invariant}
	e \triangleright k = k, \, h \triangleright e = e , \, e \triangleleft k = e , \, h \triangleleft e = h,
\end{equation}
for all $h \in H$ and $k \in K$. 
Given an admissible pair of subgroups $(H,K)$, we put $$ \Omega := \Omega(H,K) := \{ (h,k) \in H \times K \colon hk \in K H \} .$$ 

\begin{lemma} \label{lem: underlying set is open and partial actions are continuous}
	The set $\Omega = \Omega(H,K)$ is open. Moreover, the maps $$ \Omega(H,K) \to K \, , \, (h,k) \mapsto h \rt k ,$$ and $$ \Omega(H,K) \to H \, , \, (h,k) \mapsto h \lt k , $$ are continuous.
	\begin{proof}
		Since we are assuming $H K \subset G$ is open, so is $K H = (H K)^{-1}$. The set $\Omega$ is then the inverse image of the open set $H K \cap K H $ under the homeomorphism $H \times K \to H K$ given by $(h,k) \mapsto hk$.
		 
		We also have that $(k,h) \mapsto kh$ is a homeomorphism from $K \times H$ to $KH$. Therefore, since $hk = (h \rt k)(h \lt k)$, it follows that the map $\Omega \to K \times H \, , \, (h,k) \mapsto (h \rt k, h \lt k)$ is continuous, and in turn that $(h,k) \mapsto h \rt k$ and $(h,k) \mapsto h \lt k$ are continuous.
	\end{proof}
\end{lemma}

From \cref{lem: underlying set is open and partial actions are continuous} and the identities \cref{eq: extension of product in terms of H 1} - \cref{eq: extension of product in terms of K 2}, we see that both maps $(h,k) \to h \rt k$ and $(h,k) \mapsto h \lt k$ satisfy the assumptions respectively of a left local action of $H$ on $K$ and of a right local action of $K$ on $H$. Therefore, as explained in \cref{sec: prelims}, the locally compact Hausdorff space $\Omega = \Omega(H,K)$ can be endowed with two different sets of operations turning it into a locally compact Hausdorff groupoid. Let us recall these:

When endowed with the first set of operations, the multiplication between $(h_1 , k_1) \in \Omega$ and $(h_2 , k_2) \in \Omega$ is defined if and only if $k_1 = h_2 \triangleright k_2$, and then $(h_1 , h_2 \triangleright k_2) \cdot (h_2 , k_2) = (h_1 h_2 , k_2)$. The inversion is given by $(h , k)^{-1} = (h^{-1}, h \triangleright k)$. Range and source maps are given respectively by $r(h,k) = (e, h \triangleright k)$ and $s (h,k) = (e,k)$. The unit space of $\G$ is then homeomorphic to $K$, via the identification $(e,k) \mapsto k$. 

When endowed with the second set of operations, the multiplication between $(h_1 , k_1) \in \Omega$ and $(h_2 , k_2) \in \Omega$ is defined if and only if $h_2 = h_1 \lt k_1$, and then $(h_1 , k_1) \cdot (h_1 \lt k_1 , k_2) = (h_1 , k_1 k_2)$. The inversion is given by $(h , k)^{-1} = (h \lt k, k^{-1})$. Range and source maps are given respectively by $r(h,k) = (h, e)$ and $s (h,k) = (h \lt k,e)$. The unit space of $\G$ is then homeomorphic to $H$, via the identification $(h,e) \mapsto h$. 

In order to avoid confusion, let us denote by $\G$ or $\G (H,K)$ the locally compact Hausdorff groupoid obtained when endowing $\Omega$ with the first set of operations, and $\hat{\G}$ or $\hat{\G}(H,K)$ the locally compact Hausdorff groupoid obtained when endowing $\Omega$ with the second set of operations. 
Both $\G$ and $\hat{\G}$ will be referred to as the \emph{double groupoid} associated with the admissible pair of groups $(H,K)$. As mentioned in the introduction, the name stems from the fact that their underlying locally compact Hausdorff spaces admits two different groupoid structures. 

\ 

Focusing on the groupoid $\G$, let us look at some of its elementary properties. Similar results as will be obtained obviously holds for $\hat{\G}$ by reversing the roles of $H$ and $K$.  
Recall from \cref{sec: prelims} that a double groupoid $\G(H,K)$ can be seen as a partial transformation groupoid associated with the partial action of $H$ on $K$. It is interesting to know when the domains of the partial homeomorphisms are closed, since in that case, by \cite[Proposition 5.7]{Exel:PartialDynamicalSystemsFellBundlesAndApplications}, the associated partial dynamical system admits a Hausdorff globalization (see \cite[Definition 5.4]{Exel:PartialDynamicalSystemsFellBundlesAndApplications}). For a double groupoid $\G(H,K)$, this is the case when $K$ is compact:

\begin{lemma}\label{lem: domains of partial homeomrphisms are closed when the  group that is the unit space is compact}
	Let $\G(H,K)$ be a double groupoid formed from an admissible pair $(H,K)$. Let $D_h$ be the domains of the partial homeomorphisms corresponding to $h \in H$. If $HK$ is closed $G$, then $D_h$ is closed in $K$, for each $h \in H$. In particular, if $K$ is compact, then each $D_h$ is closed in $K$.
	\begin{proof}
		Fix $h \in H$. Clearly $D_h = \{k \in K \colon hk \in KH \} = h^{-1}KH \cap K $. Thus if $HK$ is closed, so is $D_h$. This is in particular the case when $K$ is compact.
	\end{proof}
\end{lemma}

The following proposition reveals the interesting fact that the inverse operation in $\hat{\G}$ is in fact an automorphism of $\G$; that is, a groupoid isomorphism from $\G$ to itself.
\begin{proposition} \label{prop: natural isomorphism associated with other action}
	The map $\gamma \colon \G \to \G \, , \, (h,k) \mapsto (h \lt k , k^{-1})$ is an automorphism of the locally compact groupoid $\G$. Moreover, $\gamma |_{\G(k)} \colon \G(k) \to \G(k^{-1}) , h \mapsto h \lt k$ is an isomorphism of topological groups.
	\begin{proof}
		We already know that $\gamma$ is a well defined continuous map. Since $\gamma^2 = Id_\G$, $\gamma$ is a homeomorphism. It remains to check that $\gamma$ is a groupoid homomorphism. To see this, notice that $$r(\gamma(h,k)) = r(h \lt k , k^{-1}) = (e, (h \lt k) \rt k^{-1}),$$ and that $$\gamma(r(h,k)) = \gamma(e, h \rt k) = (e,(h \rt k)^{-1}) .$$ Now, $$ (h \rt k)^{-1} = (h \lt k) \rt k^{-1} \iff e = (h \rt k) ((h \lt k) \rt k^{-1}) ,$$ and indeed $$ (h \rt k) ((h \lt k) \rt k^{-1}) = h \rt (k k^{-1}) = h \rt e = e ,$$ by \cref{eq: extension of product in terms of K 2}. Also, $$s(\gamma(h,k)) = s(h \lt k, k^{-1}) = (e,k^{-1})$$ and $$ \gamma (s(h,k)) = \gamma (e,k) = (e \lt k , k^{-1}) = (e, k^{-1}) .$$ It follows from this that $(\gamma(g_1) , \gamma (g_2)) \in \G^{(2)}$ whenever $(g_1 , g_2) \in \G^{(2)}$. Moreover, given $((h_1 , h_2 \rt k_2) , (h_2 , k_2)) \in \G^{(2)}$, we have that $$ \gamma ((h_1 , h_2 \rt k_2) (h_2 , k_2)) = \gamma (h_1 h_2 , k_2) = ( (h_1 h_2) \lt k_2 , k_{2}^{-1} ) ,$$ and 
		\begin{align*}
			\gamma ((h_1 , h_2 \rt k_2)) \gamma (h_2 , k_2) &= ( h_1 \lt (h_2 \rt k_2) , (h_2 \rt k_2)^{-1} ) (h_2 \lt k_2 , k_{2}^{-1}) \\
			&= ((h_1 \lt (h_2 \rt k_2)) (h_2 \lt k_2) , k_{2}^{-1}).
		\end{align*}
		The fact that $$ (h_1 h_2) \lt k_2 = (h_1 \lt (h_2 \rt k_2)) (h_2 \lt k_2) ,$$ follows from \cref{eq: extension of product in terms of H 2}, and so $\gamma$ is indeed a homomorphism. The final statement is then also clear.
	\end{proof}
\end{proposition}

The invariant subsets of the unit space can be characterized as follows.

\begin{lemma} \label{lem: characterization of invariant sets}
	A subset $A \subset K = \G^{(0)}$ is invariant if and only if $$ H A \cap K H = H A \cap A H = H K \cap A H .$$
	\begin{proof}
		Suppose $A$ is invariant. If $h \in H$ and $k \in A$ such that $hk \in K H$, then $(h,k) \in \G$ with $s(h,k) \in A$ and so we need $r(h,k) = h \rt k \in A$; that is $hk = (h \rt k) (h \lt k) \in A H.$  Similarly, if $h k \in H K \cap A H $, then $(h,k) \in \G$ and $r(h,k) = h \rt k \in A$. Again, since $A$ is invariant also $s(h,k) = k \in A$. 
		
		Conversely, suppose that we have $$ H A \cap K H = H A \cap A H = H K \cap A H. $$ If $(h,k) \in \G$ such that $s(h,k) = k \in A$, then $$hk \in H A \cap K H = H A \cap A H ,$$ so that $r(h,k) = h \rt k \in A$. Similarly, if $(h,k) \in \G$ such that $r(h,k) = h \rt k \in A$, then $$hk = (h \rt k) (h \lt k) \in A H \cap H K = H A \cap A H ,$$ so that $s(h,k) = k \in A$ also.
	\end{proof} 
\end{lemma}

\begin{corollary} \label{cor: invariance of identity and identification of H as a subgroupoid}
	The subset $\{e\} \subset K$ is invariant, and $H \cong \G_e$. 
	\begin{proof}
		\cref{lem: characterization of invariant sets} applies to give that $\{e\}$ is an invariant subset of the unit space; this also follows easily from \cref{eq: identity element is invariant}. Therefore, $\G_e = \G^e = \G(e)$ is the isotropy group at $e \in K$. Since $(h,e) \in \G$ for all $h \in H$, we have that $\G_e =\{ (h,e) \colon h \in H \}$, and the map $h \mapsto (h,e)$ is clearly an isomorphism of topological groups.
	\end{proof}
\end{corollary}

We will often identify $H$ with $\G_e$ from now on.
It follows easily from \cref{cor: invariance of identity and identification of H as a subgroupoid} that $\G = \G(H, K)$ is minimal precisely when $K$ is trivial (so $\G \cong H$ is a group), and that $\G$ is principal precisely when $H$ is trivial (so $\G \cong K$ is a locally compact Hausdorff space). Unfortunately, we have not been able to find a nice characterization for when a double groupoid is topologically principal. 

The next lemma will be important throughout, and follows readily from the work of Christensen and Neshveyev in \cite{ChristensenAndNeshveyev:IsotropyFibersOfIdealsInGroupoidCstarAlgebras} and \cite{ChristensenAndNeshveyev:NonExoticCompletionsOfTheGroupAlgebrasOfIsotropyGroups}. 
\begin{lemma} \label{lem: exact sequence of reduced groupoid C*-algebras}
	Let $\G = \G(H, K)$ be a second countable étale double groupoid constructed from an admissible pair of groups $(H,K)$. The sequence $$ 0 \to C_{r}^{\ast} (\G \setminus H) \to C_{r}^{\ast} (\G) \to C_{r}^{\ast} (H) \to 0 ,$$
	is exact, where $C_{r}^{\ast} (\G \setminus H)$ is viewed as an ideal in $C_{r}^{\ast} (\G)$ and the surjective map is the extension of the restriction map $C_c (\G) \to \C H \, , \, f \mapsto f |_{H}$.
	\begin{proof}
		By \cite[Proposition 1.2]{ChristensenAndNeshveyev:IsotropyFibersOfIdealsInGroupoidCstarAlgebras}, the sequence $$ 0 \to C_{r}^{\ast} (\G \setminus H) \to C_{r}^{\ast} (\G) \to C_{\tilde{r}}^{\ast} (H) \to 0 ,$$ is exact. Here $\| \cdot \|_{\tilde{r}}$ is the $C^*$-norm on $\C H$ given for $g \in \C H$ by $ \| g \|_{\tilde{r}} = \sup_{\pi} \| \pi (g) \|$, where the suprema runs over all representations $\pi$ of $\C H$ such that the representation $f \mapsto \pi (f |_{H})$ extends to $C_{r}^{\ast}(\G)$. In fact, by \cite[Corollary 4.15]{ChristensenAndNeshveyev:NonExoticCompletionsOfTheGroupAlgebrasOfIsotropyGroups}, $\| \cdot \|_{\tilde{r}} = \| \cdot \|_r$, so the result follows.
	\end{proof}
\end{lemma}

Recall that a locally compact Hausdorff groupoid with a Haar system is said to have the \emph{weak containment property} whenever the full and reduced $C^\ast$-norms on $C_c (\G)$ agree. Amenability always implies the weak containment property, but the reverse is false (see \cite{AlekseevAndFinnSell:NonAmenablePrincipalGroupoidsWithWeakContainment} and \cite{Willett:ANonAmenableGroupoidWhoseMaximalAndReducedC*AlgebrasAreTheSame}). For double groupoids, however, amenability is implied by the weak containment property, as is shown in the next proposition. This is a consequence of applying \cref{lem: exact sequence of reduced groupoid C*-algebras} together with results due to Renault and Williams in \cite{RenaultAndWilliams:AmenabilityOfGroupoidsArisingFromPartialSemigroupActionsAndTopologicalHigherRankGraphs}. Let us also remark that a double groupoid can be endowed with a natural Haar system (see \cite[Proposition 2.2]{Abadie:OnPartialActionsAndGroupoids}).

\begin{proposition} \label{prop: amenability equivalent to WCP for this groupoid}
	Let $\G = \G(H, K)$ be a second countable double groupoid constructed from an admissible pair of groups $(H,K)$. Consider the statements:
	\begin{itemize}
		\item[(i)] $H$ is amenable;
		\item[(ii)] $\G$ is amenable;
		\item[(iii)] $\G$ has the weak containment property.
	\end{itemize}
	Then $(i) \iff (ii) \implies (iii)$. If moreover $\G$ is étale, then they are all equivalent.
	\begin{proof}
		First of all, the implication $(i) \implies (ii)$ follows immediately from \cite[Theorem 4.2]{RenaultAndWilliams:AmenabilityOfGroupoidsArisingFromPartialSemigroupActionsAndTopologicalHigherRankGraphs} after noting that the first coordinate projection is a continuous cocycle and that its kernel $\G^{(0)}$ is an amenable groupoid. Also, since $H$ can be identified as a closed subgroupoid of $\G$, it is amenable if $\G$ is. This proves $(i) \iff (ii)$. The implication $(ii) \implies (iii)$ is well known (see for example \cite[Proposition 6.1.8]{DelarocheRenault:AmenableGroupoids}). 
		
		Suppose now that $\G$ is étale and has the weak containment property. By \cref{lem: exact sequence of reduced groupoid C*-algebras}, the sequence $$ 0 \to C_{r}^{\ast} (\G \setminus H) \to C_{r}^{\ast} (\G) \to C_{r}^{\ast} (H) \to 0 ,$$ is exact. Also, it is well known that the sequence $$ 0 \to C^{\ast} (\G \setminus H) \to C^{\ast} (\G) \to C^{\ast} (H) \to 0 ,$$ is exact (see for example \cite[Proposition 10.3.2]{Sims:EtaleGroupoids}). The diagram 
		\[
		\begin{tikzcd}
			0 \arrow{r}{} & C^{\ast} (\G \setminus H) \arrow{r}{} \arrow{d}{} & C^{\ast} (\G) \arrow{r}{} \arrow[swap]{d}{} & C^{\ast}(H) \arrow{r}{} \arrow[swap]{d}{} & 0 \\%
			0 \arrow{r}{} & C_{r}^{\ast} (\G \setminus H) \arrow{r}{} & C_{r}^{\ast} (\G) \arrow{r}{} & C_{r}^{\ast}(H) \arrow{r}{} & 0
		\end{tikzcd}
		\]
		commutes because it commutes at the pre-completed level. Here the vertical maps are induced by the identity and the horizontal respectively by the inclusion $C_c (\G \setminus H) \subset C_c (\G)$ and restriction $C_c (\G) \to \C H \, , \, f \mapsto f |_{H}$. As already mentioned, the rows in the diagram are exact, and it therefore follows by a straightforward diagram chase that $C^{\ast} (\G) \cong C_{r}^{\ast} (\G)$ imply $C^{\ast} (H) \cong C_{r}^{\ast} (H)$. By Hulanicki's theorem, $H$ is amenable.
	\end{proof}
\end{proposition}

\section{Concrete Examples} \label{sec: concrete examples}
Let us look at some concrete examples of double groupoids arising from admissible pairs of subgroups. 

\begin{example} \label{ex: transformation groupoids from Zappa-Szep products}
	Suppose $H$ and $K$ are locally compact groups and assume that there are continuous left and right actions $\phi \colon H \curvearrowright K$ and $\psi \colon H \curvearrowleft K$. For $h \in H$ and $k \in K$, we denote by $\phi_h$ and $\psi_k$ the corresponding homeomorphisms respectively on $K$ and $H$. Suppose, moreover, that these actions satisfy $$ \phi_h (k_1 k_2) = \phi_h (k_1) \phi_{\psi_{k_1}(h)}(k_2) \text{ and } \psi_k (h_1 h_2) = \psi_{\phi_{h_2} (k)} (h_1) \psi_k (h_2) ,$$ for all $h, h_1 , h_2 \in H$ and all $k , k_1 , k_2 \in K$; it is not hard to see that this forces $\phi_h (e_K) = e_K$, for all $h \in H$ and $\psi_k (e_H) = e_H$, for all $k \in K$. Let $K \bowtie H = K \times H$ be the Zappa-Szép product of $K$ and $H$. When endowed with the product topology, $K \bowtie H$ becomes a locally compact group, with identity $e_{K \bowtie H} = (e_K , e_H)$, inversion $(k,h)^{-1} = (\phi_{h^{-1}} (k^{-1}) , \psi_{k^{-1}} (h^{-1}) )$, for $h \in H$ and $k \in K$, and multiplication $(k_1 , h_1) (k_2 , h_2) = (k_1 \phi_{h_1} (k_2) , \psi_{k_2} (h_1) h_2)$, for $h_1 , h_2 \in H$ and $k_1 , k_2 \in K$. It is easy to see that $H \cong \{e_K\} \times H$, via $h \mapsto (e_K , h)$ and $K \cong K \times \{e_H\}$ via $k \mapsto (k, e_H)$, and that $K \bowtie H = KH = HK$ under these identifications. Since $$ (e_K , h) (k, e_H) = (e_k \phi_h (k) , \psi_k (h) e_H) = (\phi_h (k) , \psi_k (h)) = (\phi_h (k) , e_H) (e_K , \psi_k (h)) ,$$ we see that $h \rt k = \phi_h (k)$ and $h \lt k = \psi_k (h)$.
	Therefore, $\G = \G(H \curvearrowright K)$ and $\hat{\G} = \G(H \curvearrowleft K)$ are the transformation groupoids associated respectively with the actions $\phi \colon H \curvearrowright K$ and $\psi \colon H \curvearrowleft K$. 
\end{example}

As a very concrete example, consider the free non-Abelian group on two generators $\mathbb{F}_2$ as a subgroup of $\mathrm{SL}_2 (\C)$, the group of invertible complex matrices with determinant one. Under this identification, we get an action $\mathbb{F}_2 \to \mathrm{Aut} (\C^2)$ via matrix multiplication, and we let $\C^2$ act trivially on $\mathbb{F}_2$. Applying the above example in this situation, the double groupoid is the transformation groupoid $\G := \G(\mathbb{F}_2 , \C^2) = \mathbb{F}_2 \rtimes \C^2$, under the action of matrix multiplication, while $\hat{\G} = \hat{\G}(\mathbb{F}_2 , \C^2) = \mathbb{F}_2 \times \C^2$ is a product of groupoids, where $\mathbb{F}_2$ is considered as a discrete space. Note that $\hat{\G}$ is amenable as a product of amenable groupoids, while \cref{prop: amenability equivalent to WCP for this groupoid} gives that $\G$ does not have the weak containment property because $\mathbb{F}_2$ is non-amenable.

The next example is inspired by \cite[Section 4]{BaajSkandalisVaes:NonSemiRegularQuantumGroupsComingFromNumberTheory}.

\begin{example} \label{ex: example involving unital rings, giving possibly non-closed domains of partial homeomorphisms}
	Let $\A$ be a locally compact ring with a unit that we denote by $1$, and let $\A^*$ denote its invertible elements. Assume that $\A^*$ is open in $\A$ and that inversion in $\A$ is continuous on $\A^\ast$. We may use $\A^*$ to form the group $G:= \A^* \times \A$ that we endow with the product topology, and inversion and multiplication given by $(a, x)^{-1} = (a^{-1} , -a^{-1}x)$ and $(a,x) (b,y) = (ab , x + ay)$. The identity element is $(1,0)$. Define the two closed subgroups $$H := \{ (a , a-1) \colon a \in \A^* \},$$ and $$K := \{ (b,0) \colon b \in \A^* \} .$$ Both of these subgroups are isomorphic to $\A^*$, and as such are amenable if and only if $\A^*$ is. We have that $H \cap K = \{ (1,0) \}$ and it is easy to see that $$HK = \{ (a,x) \colon a \in \A^* \text{ and } 1 + x \in \A^* \}.$$ From this it follows that $HK$ is open; indeed, assume that $\{(a_\alpha , x_\alpha) \}_\alpha \subset G$ is a net converging to $(a,x) \in HK$. This happens if and only if $a_\alpha \to a$ and $x_\alpha \to x$. Then $1 + x_\alpha \to 1 + x \in A^*$, and since $\A^*$ is open in $\A$, it follows that there is $\alpha_0$ such that whenever $\alpha \geq \alpha_0$, we have $1 + x_\alpha \in \A^*$. This means that $(a_\alpha , x_\alpha) \in HK$ for all $\alpha \geq \alpha_0$; in other words, $HK$ is open. It need not be closed, however. It is not hard to see that the product map $H \times K \to HK \subset G$ is a homeomorphism; indeed, first of all, the product map is clearly a continuous bijection. The inverse map is given as follows: any $(a,x) \in HK$ can be written as $(x+1,x) ((x+1)^{-1} a , 0)$, where $(x+1,x) \in H$ and $((x+1)^{-1} a , 0) \in K$. Now, if $(a_\alpha , x_\alpha) \to (a,x)$ in $HK$, then $a_\alpha \to a$ and $x_\alpha \to x$, and so also $$ ( (x_\alpha +1 , x_\alpha) , ( (x_\alpha +1)^{-1} a_\alpha,0 ) ) \to ( (x+1,x), ((x+1)^{-1} a, 0) ) ,$$ showing that the inverse map is continuous also. Everything is thus in place to form the double groupoids. 
	Note that $$(a,a-1) (b,0) = (ab , a-1) = (x,0) (y, y-1) = (xy , xy - x),$$ holds if and only if $x = a(b-1) +1$ and $y = (a(b-1)+1)^{-1} ab$. From this, it follows that the underlying set of the double groupoids is $$ \Omega (H,K) = \left\lbrace \left( (a , a-1) , (b,0) \right) \colon a(b-1)+1 \in \A^* \right\rbrace  ,$$
	and that the local actions are given by $$(a,a-1) \rt (b,0) = (a (b-1) + 1  , 0),$$  and
	$$ (a , a-1) \lt (b,0) = ( (a(b-1)+1)^{-1} ab , (a(b-1)+1)^{-1} ab -1 ) ,$$  whenever $a(b-1) + 1 \in \A^*$.
\end{example}

The model example of the above is when $\A = M_2 (\K)$ and hence $\A^* = \mathrm{GL}_2(\K)$, where $\K = \R$ or $\K = \C$. 
The next example is essentially \cite[Section 6, Example 1.]{DesmedQuaegebeurVaes:AmenabilityAndTheBicrossedProductConstruction}.

\begin{example} \label{ex: example involving SL2 as one of the subgroups}
	Consider the Lie subgroup $$ G := \left\lbrace \begin{bmatrix} a & b & x \\ c & d & y \\ 0 & 0 & 1 \end{bmatrix} \colon \begin{bmatrix} a&b\\c&d \end{bmatrix} \in \mathrm{SL}_2 (\R) \text{ and } x,y \in \R \right\rbrace \leq \mathrm{SL}_3 (\R) .$$
	Let $$ H := \left\lbrace \begin{bmatrix} a & b & 0 \\ c & d & 0 \\ 0 & 0 & 1 \end{bmatrix} \colon \begin{bmatrix} a&b\\c&d \end{bmatrix} \in \mathrm{SL}_2 (\R) \right\rbrace \cong \mathrm{SL}_2 (\R) ,$$ and $$K := \left\lbrace \begin{bmatrix} 1 & 0 & -x \\ -x & 1 & -y + \frac{1}{2} x^2 \\ 0 & 0 & 1 \end{bmatrix} \colon x,y \in \R \right\rbrace \cong (\R^2 , +).$$
	Then $H \cap K = \{ I \}$, $HK \subset G$ is open with complement having Haar measure zero, so $HK$ is not closed. Moreover, $H \times K \to HK$ is a homeomorphism. So, we may form the double groupoid $\G = \G(H,K)$. Actually, it is not hard to see that $HK = KH$, so that as a locally compact Hausdorff space, $\G = H \times K = \hat{\G}$. The actions make sense everywhere, and doing the computations, one can see that these are given as follows: for any $A = \begin{bmatrix} a&b\\c&d \end{bmatrix} \in H$, and $(x,y) \in K$, we have that $$ A \rt (x,y) = \left( ax + by - \frac{1}{2}x^2 , cx + dy - \frac{d}{2} x^2 + \frac{1}{2} (ax + b(y - \frac{1}{2}x^2))^2 \right) ,$$ and $$ A \lt (x,y) = \begin{bmatrix} a-bx & b \\ c - dx + (a-bx)(ax + by - \frac{b}{2} x^2) & d + b(ax + b(y - \frac{1}{2}x^2)) \end{bmatrix} .$$ 
\end{example}

The next example is taken from \cite[Example 5.4 page 89]{VaesVainerman:ExtensionsOfLocallyCompactGroupsAndTheBicrossedProductConstruction}.

\begin{example} \label{ex: ax+b-type group example giving amenable groupoid}
	Let $G$ be the group $G := \mathrm{SL}_2 (\R) / \{ I , -I \}$, where $\{I , -I\} = \Z / 2\Z \trianglelefteq \mathrm{SL}_2 (\R)$. Then $$H := \{ (a,b) \colon a > 0, b \in \R \} \text{ with } (a,b) (c,d) = (ac , ad + \frac{b}{c}),$$ and $ K := (\R,+)$ can be identified as closed subgroups of $G$ via the embeddings $$ i(a,b) = \begin{bmatrix} a & b \\ 0 & \frac{1}{a} \end{bmatrix} \mathrm{mod} \{ I , -I \},$$ and $$ j (x) = \begin{bmatrix} 1 & 0 \\ x & 1 \end{bmatrix} \mathrm{mod} \{ I , -I \} .$$ Under these identifications, $H \cap K = \{ I \}$, $HK \subset G$ is open and not closed, and the natural map $H \times K \to HK$ is a homeomorphism. After some algebra, one can see that $$ \Omega(H,K) = \left\lbrace \left(  \begin{bmatrix} a & b \\ 0 & 1/a \end{bmatrix}  , \begin{bmatrix} 1 & 0 \\ x & 1 \end{bmatrix} \right) \colon a + bx \neq 0  \right\rbrace  .$$
	The two local actions are given in \cite[Example 5.4 page 89]{VaesVainerman:ExtensionsOfLocallyCompactGroupsAndTheBicrossedProductConstruction} as $$(a,b) \rt x = \frac{x}{a (a + bx)} \text{ and } (a,b) \lt x = \begin{cases}
		(a + bx,b) & \text{ if $a + bx > 0$} \\
		(-a -bx , -b) & \text{ if $a + bx < 0$}
	\end{cases} $$ whenever $a + bx \neq 0$.
\end{example}

\begin{example} \label{ex: example giving product of groups}
	Let $G := \mathrm{GL}_2 (\R) = \det^{-1}(\R \setminus \{0\})$, where $\det \colon M_2 (\R) \to \R$ is the determinant map. Let $H = \mathrm{SL}_2 (\R)$ and $K := \{ xI \colon x \in \R_{> 0} \} \cong (\R_{> 0} , \cdot)$; then $H,K \leq G$ are closed subgroups, $H \cap K = \{I\}$ and $HK = \det^{-1}( (0, \infty) )$ is clopen in $G$. Moreover, the map $H \times K \to HK \, (A,xI) \mapsto xA$ is a homeomorphism; indeed, we already know that it is a continuous bijection. If now $x_n A_n \to x A$ in $HK$, then taking determinants, it follows that $x_n \to x$, and so $$ A_n = x_{n}^{-1} (x_n A_n) \to x^{-1} (x A) = A ,$$ also. Thus $(A_n , x_n I) \to (A, xI)$, so that the inverse map is continuous as well. Notice that $HK = KH$, and $A \cdot x I = xI \cdot A$ for all $A \in H$ and $xI \in K$. Therefore, both actions are trivial, and the double groupoid is just $H \times K$ as a locally compact Hausdorff space. The two groupoid structures are that of a direct product of groupoids; under the structure coming from the trivial action of $H$ on $K$, it is the product of the group $H$ and the space $K$, and vice versa for the structure induced from the trivial action of $K$ on $H$.
\end{example}

\begin{example} \label{ex: free group for H and integers as K}
	Consider $H := \mathbb{F}_2$ as a subgroup of the discrete group $\mathrm{SL}_2 (\Z)$; say $$H = \left\langle \begin{bmatrix} 1 & 2 \\ 0 & 1 \end{bmatrix} , \begin{bmatrix} 1 & 0 \\ 2 & 1 \end{bmatrix} \right\rangle .$$ It was proved by Sanov in an old paper \cite{Sanov:APropertyOfARepresentationOfAFreeGroup} that $$ H = \left\lbrace \begin{bmatrix} 4n_1 + 1 & 2n_2 \\2n_3 & 4n_4 +1 \end{bmatrix} \colon n_i \in \Z \text{ and } (4n_1 +1)(4n_4 +1) - 4n_2 n_3 = 1 \right\rbrace .$$ View $H$ as a subgroup of $G := \mathrm{SL}_3 (\Z)$ via $$ H := \left\lbrace \begin{bmatrix} 4n_1 + 1 & 2n_2 & 0\\2n_3 & 4n_4 +1&0 \\ 0&0&1 \end{bmatrix} \colon n_i \in \Z \text{ and } (4n_1 +1)(4n_4 +1) - 4n_2 n_3 = 1 \right\rbrace .$$ For a 4-tuple $n \in \Z^4$ such that $(4n_1 +1)(4n_4 +1) - 4n_2 n_3 = 1$, let us write $$A_n := \begin{bmatrix} 4n_1 + 1 & 2n_2 & 0\\2n_3 & 4n_4 +1&0 \\ 0&0&1 \end{bmatrix} \in H ,$$ for the corresponding matrix . 
	Let $$K := \left\lbrace \begin{bmatrix} 1 & 0&0 \\ 0 & 1&x \\0&0&1 \end{bmatrix} \colon x \in \Z\right\rbrace \cong \Z ,$$ and let us write $$B_x = \begin{bmatrix} 1 & 0&0 \\ 0 & 1&x \\0&0&1 \end{bmatrix} \in K ,$$ for the element in $K$ corresponding to $x \in \Z$. Since all groups are considered discrete, $HK \subsetneq G$ is a clopen subset. Clearly, $H \cap K = \{I\}$ and so the product map $H \times K \to HK$ is a bijection, hence a homeomorphism.
	
	For any $A_n , A_m \in H$ and $B_x , B_y \in K$, we have that $$ A_n B_x = \begin{bmatrix} 4n_1 + 1 & 2n_2 & (2n_2)x \\2n_3 & 4n_4 +1& (4n_4 +1)x \\ 0&0&1 \end{bmatrix} ,$$ and $$ B_y A_m = \begin{bmatrix} 4m_1 + 1 & 2m_2 & 0 \\2m_3 & 4m_4 +1& y \\ 0&0&1 \end{bmatrix} .$$ Thus, the equation $A_n B_x = B_y A_m$ holds if and only if 
	\begin{itemize}
		\item $m = n$;
		\item $y = (4n_4 +1)x$;
		\item $0 = 2n_2 x$.
	\end{itemize}
	It follows from this that $$ \Omega(H,K) = \left\lbrace (A_n , B_x) \in H \times K \colon n_2 x = 0 \right\rbrace  .$$
	Given $A_n \in H$, we let $D_{A_n}$ denote the domain of the corresponding partial homeomorphism. If $n_2 = 0$, then $$D_{A_n} = \{ B_x \in K \colon A_n B_x \in KH  \} = K ,$$ whilst if $n_2 \neq 0$, then $D_{A_n} = \{B_0\} = \{I\}$. The partial action of $H$ on $K$ is given by $A_n \rt B_x = B_{(4n_4 +1)x}$, whenever this makes sense. Of course, $D_{B_0} = H$, whilst for $x \neq 0$, $D_{B_x} = \{A_n \in H \colon n_2 = 0\}$. Moreover, the above shows that the partial action of $K$ on $H$ is just $ A_n \lt B_x = A_n $, whenever this makes sense.
\end{example}

\section{Exotic $C^*$-completions of Double Groupoids} \label{sec: Existence of exotic groupoid C*-algebras}
In this section, we assume $\G = \G(H, K)$ is a second countable étale double groupoid and prove that $\G$ admit exotic $C^*$-completions of the form described in \cite{Palmstrøm:ExoticCstarCompletionsOfEtaleGroupoids} whenever $H$ admit exotic ideal completions as in \cite{BrownandGuentner:NewC*-completions}. This is at least plausible due to \cref{prop: amenability equivalent to WCP for this groupoid}, since if $H$ admit exotic $C^*$-completions, it is in particular non-amenable and so $\G(H,K)$ does not have the weak containment property. The below theorem is the main result of the paper, and gives sufficient conditions for the existence of exotic groupoid $C^*$-algebras associated with $\G$. 
\begin{theorem} \label{thm: exotic group C*-algebras induce exotic groupoid C*-algebras}
	Assume that the double groupoid $\G = \G(H, K)$ is second countable and étale. Let $\mu$ be a quasi-invariant measure with full support such that $\mu(\{e\}) = 1$, and let $D \trianglelefteq \ell^\infty (H)$ be an ideal containing $\C H$ such that $C_{D}^{\ast}(H)$ is an exotic group $C^*$-algebra. Then $$\mathcal{D} := \left\lbrace f \in B(\G) \colon f |_H \in D \text{ and } f |_{\G \setminus H} \in L^2 (\G \setminus H , \nu_\mu) \right\rbrace ,$$ is an ideal in $B(\G)$ such that $C_{\mathcal{D} , \mu}^{\ast}(\G)$ is an exotic groupoid $C^*$-algebra.
\end{theorem}

In order to prove the above theorem, we need some preparation. 
First of all, it may not be immediately obvious that such a quasi-invariant measure should exist. The next lemma shows that indeed it does.
\begin{lemma} \label{lem: there exists a quasi-invariant measure with full support}
	For any second countable Hausdorff étale groupoid $\G$ there exists a quasi-invariant measure on $\G^{(0)}$ with full support. If $\{u\} \subset \G^{(0)}$ is invariant, then we may choose the quasi-invariant measure such that it takes the value one on the set $\{u\}$. In that case, the associated modular function restricted to the isotropy group $\G(u)$ is identically one.
	\begin{proof}
		By \cite[Theorem 1.21, page 65]{Renault:AGroupoidApproachToC*Algebras}, there exists a unitary representation of $\G$, $\pi$, and a quasi-invariant measure $\mu$, such that $\| f \|_{max} = \| \pi_\mu (f) \|$, for all $f \in C_c (\G)$. We claim that $\mu$ has full support. Indeed, if $v \in \G^{(0)}$ is a unit such that $v \notin \supp (\mu)$, then there is some compact neighborhood $V$ of $v$ in $\G^{(0)}$ such that $\mu(V) = 0$. Let $\chi_V$ be in $C_c (\G^{(0)})$ such that $\supp (\chi_V) \subset V$, $\chi_V (v) = 1$ and $0 \leq \chi_V \leq 1$. Then $1 = \Vert \chi_V \Vert_{\infty} = \Vert \chi_V \Vert_{max} = \Vert \pi_\mu (\chi_V) \Vert$. At the same time, for any pair of bounded square-integrable sections $\xi , \eta \in \HS_{\pi , \mu}$, we have that 
		\begin{align*}
			\left| \langle \pi_{\mu} (\chi_V) \xi , \eta \rangle \right| &= \left| \int_{\G^{(0)}} \sum_{x \in \G^{u}} \chi_V (x) \langle \pi(x) \xi(s(x)) , \eta(r(x)) \rangle_{\HS_\pi (r(x))} \Delta^{-1/2}(x) \, d \mu(u) \right| \\ &\leq \int_{\G^{(0)}} \chi_V (u) \Vert \xi(u) \Vert_{\HS_\pi (u)} \Vert \eta(u) \Vert_{\HS_\pi (u)} \, d \mu(u) = 0.
		\end{align*}
		Since bounded square-integrable sections are dense in $\HS_{\pi , \mu}$, we need $\pi_\mu (\chi_V) = 0$, which is impossible. Therefore, we may conclude that $\mu$ has full support. 
		Suppose that $\{u\} \subset \G^{(0)}$ is an invariant subset. If we do not have $\mu(\{u\}) > 0$ already, then $\tilde{\mu} := \mu + \mu_u$ is again quasi-invariant with full support; here $\mu_u$ is the Dirac measure at the unit $u$. But now $\tilde{\mu}$ also satisfies $\tilde{\mu} (\{u\}) > 0$. In either case, normalize to obtain a quasi-invariant Radon measure $\mu$ with full support such that $\mu (\{u\}) = 1$. 
		
		Let us show that the associated modular function $\Delta := \frac{d \nu_\mu}{d \nu_{\mu}^{-1}}$ is identically equal to the constant function one at the isotropy group $\G(u)$. We know that $\Delta$ satisfies $$ \int_\G 1_E (x) \Delta(x) \, d \nu_{\mu}^{-1} (x) = \int_\G 1_E (x) \, d \nu_\mu (x) ,$$ for all Borel sets $E \subset \G$, and for such Borel sets, we have by \cite[Corollary 3.9]{Williams:AToolKitForGroupoidCstarAlgebras} that
		$$ \int_\G 1_E (x) \, d \nu_\mu (x) = \int_{\G^{(0)}} \sum_{x \in \G^v} 1_E (x) \, d \mu(v) ,$$ and similarly
		$$ \int_\G 1_E (x) \Delta(x) \, d \nu_{\mu}^{-1} (x) = \int_{\G^{(0)}} \sum_{x \in \G_v} 1_E (x) \Delta (x) \, d \mu(v) .$$
		In particular, since $\mu(\{u\}) = 1$, we have that $$ \int_\G 1_E (x) \Delta(x) \, d \nu_{\mu}^{-1} (x) = \sum_{x \in \G(u)} 1_E (x) \Delta (x) = \sum_{x \in \G(u)} 1_E (x) = \int_\G 1_E (x) \, d \nu_\mu (x) , $$ for all $E \subset \G(u)$ Borel. This forces $\Delta = 1$ on $\G(u)$.
	\end{proof} 
\end{lemma}

If a closed or open subset $X \subset \G^{(0)}$ is invariant, then any unitary representation of the groupoid $\sigma \colon \G \to \text{Iso}(\G^{(0)} \ast \HS_\sigma)$, canonically restricts to a unitary representation of the groupoid $\G(X)$, $\sigma |_{\G(X)} \colon \G(X) \to \text{Iso}(X \ast \HS_\sigma)$. The next lemma states a form of converse to this.
\begin{lemma} \label{lem: representations lifts}
	Let $\G$ be a second countable étale groupoid. 
	Suppose that $X \subset \G^{(0)}$ is a closed or open invariant subset and let $\pi \colon \G(X) \to \text{Iso}(X \ast \HS_\pi)$ and $\rho \colon \G(X^c) \to \text{Iso}(X^c \ast \HS_\rho)$ be unitary representations. Then there is a unitary representation $\sigma \colon \G \to \text{Iso}(\G^{(0)} \ast \HS_\sigma)$ such that $\sigma \big|_{\G(X)} = \pi$ and $\sigma \big|_{\G(X^c)} = \rho$.
	\begin{proof}
		If $X \ast \HS_\pi$ and $X^c \ast \HS_\rho$ are the Borel Hilbert bundles corresponding respectively to $\pi$ and $\rho$, then we may form the Hilbert bundle $\G^{(0)} \ast \HS_\sigma$ where $\HS_\sigma(x) := \HS_{\pi}(x)$, for all $x \in X$, and $\HS_\sigma(x) := \HS_\rho(x)$, for all $x \in X^c$. If $\{f_n\}_{n}$ and $\{g_m \}_m$ are fundamental sequences for $X \ast \HS_\pi$ and $X^c \ast \HS_\rho$ respectively, then we extend these in such a way that $f_n (x) = 0 \in \HS_\rho (x)$, whenever $x \notin X$, and $g_m (x) = 0 \in \HS_\pi (x)$, whenever $x \in X$. The sequence given by the union of these two satisfies property (ii) (b) and (ii) (c) of the definition of a Borel Hilbert bundle in \cref{sec: prelims}, and hence by \cite[Proposition 3.2]{Muhly:CoordinatesInOperatorAlgebras} there exists a unique standard Borel structure on the Hilbert bundle $\G^{(0)} \ast \HS_\sigma$ such that the sequence $\{f_n\}_n \cup \{g_m\}_m$ is a fundamental sequence.
		Define a representation $\sigma \colon \G \to \text{Iso}(\G^{(0)} \ast \HS_\sigma)$ by $\sigma (\gamma) = \pi (\gamma)$, if $\gamma \in \G(X)$ and $\sigma (\gamma) = \rho(\gamma) $, if $\gamma \in \G(X^c)$. Then $\sigma \colon \G \to \text{Iso}(\G^{(0)} \ast \HS_\sigma)$ is a Borel homomorphism of groupoids which by definition satisfies $\sigma |_{\G(X)} = \pi$ and $\sigma |_{\G(X^c)} = \rho$.
	\end{proof}
\end{lemma}

\begin{lemma} \label{lem: restriction map extends to the corresponding completions}
	Suppose that $\G = \G(H,K)$ is a second countable étale double groupoid formed from an admissible pair $(H,K)$. 
	Let $\psi \colon C_c (\G) \to \C H$ denote the restriction $\psi(f) = f|_H$, and let $\mathcal{D}$ and $D$ be ideals as in \cref{thm: exotic group C*-algebras induce exotic groupoid C*-algebras}. Let $\mu$ be a quasi-invariant measure on $\G^{(0)}$ with full support such that $\mu (\{e\}) = 1$. Then $\psi$ extends to a surjective $C^*$-homomorphism $\psi \colon C_{\mathcal{D}, \mu}^{\ast} (\G) \to C_{D}^{\ast} (H)$.
	\begin{proof}
		Recall that for $f \in \C H$, the $D$-norm was given by $\| f \|_{D} = \sup_{\pi} \| \pi (f) \|$, where the suprema is taken over all $D$-representations $\pi$ of $H$. The $\mathcal{D}, \mu$-norm was given for $f \in C_c (\G)$ by $\| f \|_{D, \mu} = \sup_{\pi} \| \pi_\mu (f) \|$, where the suprema is taken over all $\mathcal{D}$-representations $\pi$ of $\G$.
		
		If $\pi$ is a $D$-representation, then as $\{e\}$ is an invariant closed subset of $K = \G^{(0)}$, we may by \cref{lem: representations lifts} extend $\pi$ to a unitary representation $\tilde{\pi}$ such that $\tilde{\pi}|_{H} = \pi$ by letting the other representation therein be the left regular representation of $\G \setminus H$. With the particular fundamental sequence of the resulting Borel Hilbert bundle $K \ast \HS_{\tilde{\pi}}$ described in that lemma, we see that $\tilde{\pi}$ is a $\mathcal{D}$-representation. Now,
		\begin{align*}
			\| \psi(f) \|_{D} &= \| f |_{H} \|_{D} \\
			&= \sup_{\pi} \| \pi(f|_H) \| \\
			&= \sup_{\pi} \sup_{\substack{\xi , \eta \in \HS_\pi \\ \| \xi\|_2 , \|\eta\|_2 \leq 1}} | \langle \pi (f|_H) \xi , \eta \rangle | \\
			&= \sup_{\pi} \sup_{\substack{\xi , \eta \in \HS_\pi \\ \| \xi\|_2 , \|\eta\|_2 \leq 1}} | \langle \tilde{\pi}_{\mu} (f) \tilde{\xi} , \tilde{\eta} \rangle | \\
			&\leq \sup_{\pi} \sup_{\substack{\tilde{\xi} , \tilde{\eta} \in \HS_{\tilde{\pi} , \mu} \\ \| \tilde{\xi} \|_2 , \| \tilde{\eta} \|_2 \leq 1}} | \langle \tilde{\pi}_{\mu} (f) \tilde{\xi} , \tilde{\eta} \rangle | \\ 
			&= \sup_{\pi} \| \tilde{\pi}_{\mu} (f) \| \leq \| f \|_{\mathcal{D}, \mu},
		\end{align*}
		where in the fourth equality, by for example $\tilde{\xi}$, we mean the section given by $\tilde{\xi} (k) := 0$ when $k \neq e$ and $\tilde{\xi}(e) := \xi$. It follows that $\psi \colon C_c (\G) \to \C H$ extends to a surjective $C^*$-homomorphism $\psi \colon C_{\mathcal{D} , \mu}^{\ast}(\G) \to C_{D}^{\ast}(H)$.
	\end{proof}
\end{lemma}

\begin{proof} [Proof of \cref{thm: exotic group C*-algebras induce exotic groupoid C*-algebras}]
	First of all, to show that $C_{\mathcal{D} , \mu}^{\ast}(\G) \neq C_{r}^{\ast}(\G)$, recall that by \cref{lem: exact sequence of reduced groupoid C*-algebras}, we have an exact sequence $$ 0 \to C_{r}^{\ast} (\G \setminus H) \to C_{r}^{\ast} (\G) \to C_{r}^{\ast} (H) \to 0 .$$ It follows by \cref{lem: restriction map extends to the corresponding completions} that the sequence $$ C_{\mathcal{D} , \mu}^{\ast} (\G \setminus H) \hookrightarrow C_{\mathcal{D} , \mu}^{\ast} (\G) \twoheadrightarrow C_{D}^{\ast} (H) ,$$ is a chain complex, where $C_{\mathcal{D} , \mu}^{\ast} (\G \setminus H) \trianglelefteq C_{\mathcal{D} , \mu}^{\ast} (\G)$ is the ideal given by the completion of $C_c (\G \setminus H) \subset C_{\mathcal{D} , \mu}^{\ast} (\G)$. Moreover, the sequence $$ 0 \to C^{\ast} (\G \setminus H) \to C^{\ast} (\G) \to C^{\ast}(H) \to 0 ,$$ is exact, and so we have the following commutative diagram
	\[
	\begin{tikzcd}
		0 \arrow{r}{} & C^{\ast} (\G \setminus H) \arrow{r}{} \arrow{d}{} & C^{\ast} (\G) \arrow{r}{} \arrow[swap]{d}{} & C^{\ast}(H) \arrow{r}{} \arrow[swap]{d}{} & 0 \\
		 & C_{\mathcal{D} , \mu}^{\ast} (\G \setminus H) \arrow[hookrightarrow]{r} \arrow{d}{} & C_{\mathcal{D} , \mu}^{\ast} (\G) \arrow[twoheadrightarrow]{r} \arrow[swap]{d}{} & C_{D}^{\ast}(H) \arrow[swap]{d}{} &  \\%
		0 \arrow{r}{} & C_{r}^{\ast} (\G \setminus H) \arrow{r}{} & C_{r}^{\ast} (\G) \arrow{r}{} & C_{r}^{\ast}(H) \arrow{r}{} & 0
	\end{tikzcd}
	\]
	where the bottom and top row are exact, and the middle row is a chain complex. If $C_{\mathcal{D}, \mu}^\ast (\G) = C_{r}^{\ast}(\G)$, also $C_{\mathcal{D} , \mu}^{\ast} (\G \setminus H) = C_{r}^{\ast} (\G \setminus H)$ by definition, and it follows by a straightforward diagram-chase that $C_{D}^\ast (H) = C_{r}^{\ast} (H)$, which is a contradiction; so we need $C_{\mathcal{D} , \mu}^{\ast}(\G) \neq C_{r}^{\ast}(\G)$. 
	
	Next, suppose for a contradiction that $C^\ast (\G) = C_{\mathcal{D}, \mu}^{\ast} (\G)$. We cannot argue in the same way as the above to the topmost square in the commutative diagram, but nevertheless the contradiction we shall obtain is similarly that $C^\ast (H) = C_{D}^{\ast} (H)$. In any case, by \cite[Proposition 3.13]{Palmstrøm:ExoticCstarCompletionsOfEtaleGroupoids} $C^\ast (\G)$ has a faithful $\mathcal{D}$-representation with respect to $\mu$, say $\pi$; so $\Vert f \Vert = \Vert \pi_\mu (f) \Vert$, for all $f \in C_c (\G)$. Denote by $\psi$ the map in the topmost right corner; that is, the map $\psi \colon C_{\mathcal{D}, \mu}^{\ast} (\G) \to C^{\ast} (H) $ which extends the restriction map on $C_c (\G)$. Let $I:= C_{\mathcal{D} , \mu}^{\ast} (\G \setminus H) = C^{\ast} (\G \setminus H)$. Then $\psi$ induces an isomorphism $$\tilde{\psi} \colon C_{\mathcal{D}, \mu}^{\ast} (\G)/I \to C^{\ast} (H) ,$$ where for any $f \in \C H$, we have that $\tilde{\psi}(\tilde{f} + I) = f$, for any $\tilde{f} \in C_c (\G)$ such that $\tilde{f}|_H = f$. In fact, fixing $f \in \C H$, we have that $$ \Vert f \Vert = \Vert \tilde{\psi}(\tilde{f} + I) \Vert = \Vert \tilde{f} + I \Vert = \inf_{a \in I} \Vert \pi_\mu (\tilde{f} + a) \Vert \leq \inf_{\tilde{f}|_H = f} \Vert \pi_\mu (\tilde{f}) \Vert .$$ 
	Let $g \in C_c (\G)$ be such that $g |_H = f$. By \cite[Lemma 3.3]{Palmstrøm:ExoticCstarCompletionsOfEtaleGroupoids}, there exists a dense subspace of $\HS_{\pi, \mu}$, say $\mathscr{K}$, such that for all $\xi, \eta \in \mathscr{K}$, we have that the map $$ x \mapsto \langle \pi(x) \xi (s(x)) , \eta(r(x)) \rangle_{\HS_\pi (r(x))} ,$$ is in $\mathcal{D}$. 
	Arguing as in the first part of the proof in \cite[Lemma 3.14]{Palmstrøm:ExoticCstarCompletionsOfEtaleGroupoids}, we obtain that
	\begin{align*}
		\Vert \pi_\mu (g) \Vert &= \sup_{\substack{\xi \in \mathscr{K} \\ \Vert \xi \Vert_2 \leq 1 }} \lim_{n \to \infty} \langle \pi_\mu ( (g^\ast \ast g)^{\ast 2n}) \xi , \xi \rangle^{1/4n} \\
		&= \sup_{\substack{\xi \in \mathscr{K} \\ \Vert \xi \Vert_2 \leq 1 }} \lim_{n \to \infty} \left| \int_\G (g^\ast \ast g)^{\ast 2n}(x) \Delta_{\mu}^{-1/2} (x) \langle \pi(x) \xi(s(x)) , \xi(r(x)) \rangle_{\HS_{\pi} (r(x))} \, d \nu_\mu (x) \right|^{1/4n} \\
		&= \sup_{\substack{\xi \in \mathscr{K} \\ \Vert \xi \Vert_2 \leq 1 }} \limsup_{n \to \infty} \left| \int_\G (g^\ast \ast g)^{\ast 2n}(x) \Delta_{\mu}^{-1/2} (x) \langle \pi(x) \xi(s(x)) , \xi(r(x)) \rangle_{\HS_{\pi} (r(x))} \, d \nu_\mu (x) \right|^{1/4n}.
	\end{align*}
	Since $\mu(\{e\}) = 1$, $\Delta_\mu = 1$ on $H$ and $\pi$ is a $\mathcal{D}$-representation, we get that
	\begin{align*}
		&\sup_{\substack{\xi \in \mathscr{K} \\ \Vert \xi \Vert_2 \leq 1 }} \limsup_{n \to \infty} \left| \int_\G (g^\ast \ast g)^{\ast 2n}(x) \Delta_{\mu}^{-1/2} (x) \langle \pi(x) \xi(s(x)) , \xi(r(x)) \rangle_{\HS_{\pi} (r(x))} \, d \nu_\mu (x) \right|^{1/4n} \\
		&\leq \sup_{\substack{\xi \in \mathscr{K} \\ \Vert \xi \Vert_2 \leq 1 }} \limsup_{n \to \infty} \left| \sum_{h \in H} (f^\ast \ast f)^{\ast 2n}(h) \langle \pi|_H (h) \xi(e) , \xi(e) \rangle_{\HS_{\pi} (e)} \right|^{1/4n} \\
		&+ \sup_{\substack{\xi \in \mathscr{K} \\ \Vert \xi \Vert_2 \leq 1 }} \limsup_{n \to \infty} \left| \int_{\G \setminus H} (g^\ast \ast g)^{\ast 2n}(x) \Delta_{\mu}^{-1/2} (x) \langle \pi(x) \xi(s(x)) , \xi(r(x)) \rangle_{\HS_{\pi} (r(x))} \, d \nu_\mu (x) \right|^{1/4n} \\
		&\leq \Vert \pi|_H (f) \Vert + \sup_{\substack{\xi \in \mathscr{K} \\ \Vert \xi \Vert_2 \leq 1 }} \limsup_{n \to \infty} \left| \int_{\G \setminus H} (g^\ast \ast g)^{\ast 2n}(x) \Delta_{\mu}^{-1/2} (x) \langle \pi(x) \xi(s(x)) , \xi(r(x)) \rangle_{\HS_{\pi} (r(x))} \, d \nu_\mu (x) \right|^{1/4n} \\
		&\leq \Vert f \Vert_{D} + \sup_{\substack{\xi \in \mathscr{K} \\ \Vert \xi \Vert_2 \leq 1 }} \limsup_{n \to \infty} \left| \int_{\G \setminus H} (g^\ast \ast g)^{\ast 2n}(x) \Delta_{\mu}^{-1/2} (x) \langle \pi(x) \xi(s(x)) , \xi(r(x)) \rangle_{\HS_{\pi} (r(x))} \, d \nu_\mu (x) \right|^{1/4n}.
	\end{align*} 
		Let us estimate the latter term: Fix any $\xi \in \mathscr{K}$ with $\Vert \xi \Vert_2 \leq 1$. Then
	\begin{align*}
		&\limsup_{n \to \infty} \left| \int_{\G \setminus H} (g^\ast \ast g)^{\ast 2n}(x) \Delta_{\mu}^{-1/2} (x) \langle \pi(x) \xi(s(x)) , \xi(r(x)) \rangle_{\HS_{\pi} (r(x))} \, d \nu_\mu (x) \right|^{1/4n} \\
		&\leq \limsup_{n \to \infty} \Vert (g^\ast \ast g)^{\ast 2n} \Vert_{L^2 (\G \setminus H , \nu_{\mu}^{-1})}^{1/4n} \cdot \limsup_{n \to \infty} \Vert \langle \pi(\cdot) \xi(s(\cdot)) , \xi(r(\cdot)) \rangle_{\HS_{\pi} (r(\cdot))} \Vert_{L^2 (\G \setminus H , \nu_\mu)}^{1/4n} \\
		&= \limsup_{n \to \infty} \Vert (g^\ast \ast g)^{\ast 2n} \Vert_{L^2 (\G \setminus H , \nu_{\mu}^{-1})}^{1/4n} \\
		&\leq \limsup_{n \to \infty} \Vert (g^\ast \ast g)^{\ast 2n} \Vert_{L^2 (\G , \nu_{\mu}^{-1})}^{1/4n} = \lim_{n \to \infty} \Vert (g^\ast \ast g)^{\ast 2n} \Vert_{L^2 (\G, \nu_{\mu}^{-1})}^{1/4n} = \Vert \mathrm{Ind}(\mu) (g) \Vert = \Vert g \Vert_r ,
	\end{align*}
		where in the first inequality we have used Cauchy-Schwartz, in the second the fact that the function $\langle \pi(\cdot) \xi(s(\cdot)) , \xi(r(\cdot)) \rangle_{\HS_{\pi} (r(\cdot))}$ restricted to $\G \setminus H$ is in $L^2 (\G \setminus H, \nu_\mu)$, and in the sixth the fact that $\mu$ has full support. 
		So, we may conclude that $$ \Vert f \Vert \leq \Vert f \Vert_{D} + \Vert g \Vert_r , $$ for any $g \in C_c (\G)$ such that $g |_H = f$. Combining \cite[Theorem 2.4]{ChristensenAndNeshveyev:NonExoticCompletionsOfTheGroupAlgebrasOfIsotropyGroups} with \cite[Corollary 4.15]{ChristensenAndNeshveyev:NonExoticCompletionsOfTheGroupAlgebrasOfIsotropyGroups}, we see that 
		\begin{equation} \label{eq: full norm dominated by constant times DH norm}
			\Vert f \Vert \leq \Vert f \Vert_{D} + \inf_{g |_H = f} \Vert g \Vert_r = \Vert f \Vert_{D} + \Vert f \Vert_{r} \leq 2\Vert f \Vert_{D},
		\end{equation}
		for all $f \in \C H$. If $\rho \colon C^\ast (H) \to C_{D}^{\ast} (H)$ denotes the canonical non-injective surjection induced by the identity on $\C H$, then it follows from \cref{eq: full norm dominated by constant times DH norm} that $ \Vert a \Vert \leq 2 \Vert \rho(a) \Vert_{D} $ for all $a \in C^\ast (H)$, which forces $\ker \rho = 0$, a contradiction. In summary, the canonical surjections $$C^\ast (\G) \twoheadrightarrow C_{\mathcal{D}, \mu}^\ast (\G) \twoheadrightarrow C_{r}^\ast (\G) ,$$ are non-injective, and therefore $C_{\mathcal{D}, \mu}^\ast (\G)$ is an exotic groupoid $C^*$-algebra.
\end{proof}

In fact, a straightforward alteration to the the second part of the proof of \cref{thm: exotic group C*-algebras induce exotic groupoid C*-algebras} shows the following. 

\begin{proposition} \label{prop: a possible contiuum of exotic groupoid Cstar algebras}
	Let $\G = \G(H,K)$ be a second countable étale double groupoid associated to an admissible pair $(H,K)$.
	Suppose $ D_1 \subsetneq D_2 \trianglelefteq \ell^\infty (H)$ are ideals containing $\C H$ such that the canonical surjection $C_{D_2}^{\ast} (H) \to C_{D_1}^{\ast} (H) $ is non-injective. Then with the ideals $\mathcal{D}_1 \subsetneq \mathcal{D}_2 \trianglelefteq B(\G)$ and a full support quasi-invariant measure $\mu$ such that $\mu (\{e\}) = 1$, the canonical surjection $ C_{\mathcal{D}_2 , \mu}^{\ast}(\G) \to C_{\mathcal{D}_1 , \mu}^{\ast}(\G) $ is non-injective.
\end{proposition}

\begin{example} \label{ex: example of double groupoid admitting exotic C*-completions under one structure and no exotic C*-completions under the other}
	Consider \cref{ex: free group for H and integers as K} where $H = \mathbb{F}_2$ and $K = \Z$. By \cref{prop: amenability equivalent to WCP for this groupoid}, since $K$ is amenable, so is $\hat{\G} = \hat{\G}(H,K)$, and in particular, $C_c (\hat{\G})$ does not admit any exotic $C^*$-norms. At the same time, Okayasu proves in \cite[Corollary 3.2]{Okayasu:FreeGroupC*-algebrasAssociatedWithLP} that with the ideals $\ell^p (H)$ the $C^*$-completions $C_{\ell^p (H)}^{\ast} (H)$, for $p \in (2,\infty)$, are distinct exotic group $C^*$-algebras associated to $H$.  If $\mu$ denotes a full support quasi-invariant measure on $K$ such that $\mu(\{e\}) = 1$ and $\mathcal{D}_p$ is the ideal corresponding to $\ell^p (H)$ as in the statement of \cref{thm: exotic group C*-algebras induce exotic groupoid C*-algebras}, then \cref{prop: a possible contiuum of exotic groupoid Cstar algebras} says that for $p \in (2, \infty)$, $C_{\mathcal{D}_p , \mu}^{\ast} (\G)$ are distinct exotic groupoid $C^*$-algebra associated to $\G$. So, $C_c (\G)$ admit (a continuum of) exotic $C^*$-completions, while $C_c (\hat{\G})$ does not admit any. 
	Similarly, we can consider \cref{ex: transformation groupoids from Zappa-Szep products} with $K = \R^2$ and $H = \mathrm{SL}_2(\Z)$, $H$ acting on $K$ by matrix multiplication and $K$ acting trivially on $H$. Wiersma proved in \cite{Wiersma:ConstructionsOfExoticGroupC*Algebras} that $C_{\ell^p (H)}^{\ast} (H)$, for $p \in (2, \infty)$, are distinct exotic group $C^*$-algebras associated to $H$. So, by \cref{prop: a possible contiuum of exotic groupoid Cstar algebras}, we may conclude that the transformation groupoid $\G(H,K) = \mathrm{SL}_2(\Z) \rtimes \R^2$ admits a continuum of exotic $C^*$-completions, while the product groupoid $\hat{\G}(H,K) = \mathrm{SL}_2(\Z) \times \R^2$, where $\mathrm{SL}_2(\Z)$ is considered as a discrete space and $\R^2$ as a group, is amenable and therefore does not admit any exotic completions.
\end{example}

\printbibliography

\end{document}